\documentclass[11pt]{amsart}
 \usepackage[margin=3.7cm]{geometry}
\numberwithin{equation}{section}
\usepackage{amssymb}
\usepackage{amsmath, amsfonts,amsthm,amssymb,amscd, verbatim,graphicx,color,multirow,booktabs, caption,tikz,tikz-cd, mathdots,bm}
\usepackage{tikz-cd}
 \usepackage{hyperref} 
\usetikzlibrary{positioning}

\newtheorem{theorem}{Theorem}[section]
\newtheorem{lemma}[theorem]{Lemma}

\newtheorem{proposition}[theorem]{Proposition}

     \newtheorem{question}{Question}

      \theoremstyle{definition}
     \newtheorem*{definition}{Definition}
     \newtheorem{example}[theorem]{Example}
     
     \theoremstyle{remark}
     \newtheorem{remark}[theorem]{Remark}

\newcommand{\Aut}{\mathop{\mathrm{Aut}}}

\newcommand\blfootnote[1]{%
  \begingroup
  \renewcommand\thefootnote{}\footnote{#1}%
  \addtocounter{footnote}{-1}%
  \endgroup
}

 \definecolor{mycolor}{rgb}{0.55,0.0,0.16}
  \definecolor{myred}{rgb}{0.6,0.0,0.16}
  \definecolor{mygreen}{rgb}{0.0,0.6,0.16}
  \definecolor{myviolet}{rgb}{1,0,1}

\hypersetup{
colorlinks=true,
linkcolor=true,
linktocpage=true,
pageanchor=true,
hyperindex=true
}

 \AtBeginDocument{
     \hypersetup{
  linkcolor=myred,
  urlcolor=mycolor,
citecolor=mygreen
}
     }

\makeatletter
\@namedef{subjclassname@2020}{%
  \textup{2020} Mathematics Subject Classification}
\makeatother

\begin{document}

  \title[On groups with large verbal quotients]{On groups with large verbal quotients}

\author[Lisi and Sabatini]{Francesca Lisi and Luca Sabatini}

\address{Francesca Lisi, Universit\`a degli Studi di Firenze} 
\email{francesca.lisi@unifi.it}

\address{Luca Sabatini, Alfr\'ed R\'enyi Institute of Mathematics} 
\email{sabatini@renyi.hu, sabatini.math@gmail.com}

	\maketitle

	  \begin{abstract}
	  Let $w=w(x_1,...,x_n)$ be a word, i.e. an element of the free group $F = \langle x_1,...,x_n \rangle$.
	The verbal subgroup $w(G)$ of a group $G$ is the subgroup generated by the set $\{ w(x_1,...,x_n) : x_1,...,x_n \in G \}$ of all $w$-values in $G$.  
	  Following J. Gonz\'alez-S\'anchez and B. Klopsch,
	   a group $G$ is $w$-maximal if $|H:w(H)| < |G:w(G)|$ for every $H<G$.
	  In this paper we give new results on $w$-maximal groups,
	  and study the weaker condition in which the previous inequality is not strict.
	 Some applications are given: for example, if a finite group has a solvable (resp. nilpotent)
	 section of size $n$, then it has a solvable (resp. nilpotent) subgroup of size at least $n$.
          \end{abstract}
	
	\blfootnote{This work is dedicated to our professor Carlo Casolo.}

	\begin{section}{Introduction}
	
	Let $p$ be a prime.
	In~\cite{1969Thompson}, J.G. Thompson observed that, if $G$ is a finite $p$-group such that $|H:[H,H]|<|G:[G:G]|$ for every $H<G$,
	then the nilpotency class of $G$ is at most $2$.
	In~\cite{2022Sab}, the second author remarks that the previous strict inequality provides nilpotency,
	and so the hypothesis that $G$ is a $p$-group can be removed.
	With a more general approach, J. Gonz\'alez-S\'anchez and B. Klopsch introduced the concept of $w$-maximal group~\cite{2011GK}.
	Let $w=w(x_1,...,x_n)$ be a word, i.e. an element of the free group $F = \langle x_1,...,x_n \rangle$.
	The verbal subgroup $w(G)$ of a group $G$ is the subgroup generated by the set $\{ w(x_1,...,x_n) : x_1,...,x_n \in G \}$ of all $w$-values in $G$.
	
	\begin{definition}
	A group $G$ is {\itshape $w$-maximal} if $|H:w(H)| < |G:w(G)|$ for every $H<G$.
	\end{definition}
	
	They focus on finite $p$-groups, and prove $w(G) \leqslant Z(G)$ for a large class of words.
	The present article has two goals: give new and more general results on $w$-maximal groups,
	and study the following weaker condition.
	
	\begin{definition}
	A group $G$ is {\itshape weakly $w$-maximal} if $|H:w(H)| \leq |G:w(G)|$ for every $H \leqslant G$.
	\end{definition}
	
	As in~\cite{2011GK}, the bulk of the paper is about finite groups.
	We show that this is not really restrictive, because every residually finite weakly $w$-maximal group is in fact finite in many important cases.
	The general approach naturally leads to the following useful result.
	
	\begin{proposition} \label{propSolvNilpSec}
	 Let $G$ be a finite group.
	 If $G$ has a solvable (resp. nilpotent) section of size $n$, then it has a solvable (resp. nilpotent) subgroup of size at least $n$.
	 \end{proposition}
	
	Let $n \geq 2$ and $\gamma_n := [x_1,...,x_n]$,
	so that $\gamma_2$-maximal groups are those involved in the original Thompson's theorem.
	Unexpectedly, the hypothesis of weak $\gamma_2$-maximality is sufficient to obtain nilpotency,
	but does not provide any bound on the nilpotency class.
	Moreover, the same is true when $n=3$.
	(We stress that the properties of being weakly $\gamma_{n}$-maximal
	and weakly $\gamma_{n+1}$-maximal are independent.)
	
	\begin{theorem} \label{thMain}
	Let $n=2,3$.
	Every weakly $\gamma_n$-maximal finite group is the direct product of weakly $\gamma_n$-maximal finite $p$-groups.
	\end{theorem}
	 
	In some sense, Theorem \ref{thMain} is the best possible, because it fails for each $n \geq 4$ (see Example \ref{exContro}).
	 Our proof does not require the classification of the non-abelian finite simple groups.\\
	 In Section \ref{sec3}, we face the problem of studying the $w$-maximal groups more in depth.
	We prove the following general theorem.
	 
	  \begin{theorem} \label{thIso}
	The properties of being weakly $w$-maximal and $w$-maximal are invariant under $w$-isologism.
	\end{theorem}
	 
	 We refer to Section \ref{sec3} for the detailed statement.
	 As we will explain, Theorem \ref{thIso} helps to better understand weakly $w$-maximal and $w$-maximal groups,
	 especially when $w=\gamma_2$.
	 We conclude with a short list of questions.
	
	\end{section}

	\vspace{0.1cm}
	\begin{section}{General results}
	
	We start with some basic facts concerning verbal subgroups.
	
	\begin{lemma} \label{lemHelp}
	Let $w$ be a word, and $N \lhd G$.
	Then
	\begin{itemize}
	   \item[(i)] $w(G/N)=Nw(G)/N$;
	  \item[(ii)] $|G:w(G)| \leq |N:w(N)| |(G/N):w(G/N)|$.
	\end{itemize}
	\end{lemma}
	\begin{proof}
	(i) is trivial, and for (ii) we use the equality
	\begin{align*} 
|G:w(G)| & \> = \> 
|G:Nw(G)| |Nw(G):w(G)| \\ & \> = \> 
   |(G/N):w(G/N)| |N : N \cap w(G)| \> ,
\end{align*}
	and the fact that $w(N) \leqslant N \cap w(G)$.
	\end{proof}
	 
	 \begin{remark} \label{remDP}
	 If $A \times B$ is the direct product of two groups, then $w(A \times B)=w(A) \times w(B)$.
	 \end{remark}
	 
	 We will frequently use the following computation:
	 if $G=Hw(G)$ for some $H \leqslant G$, then
	 \begin{align} \label{eqTop}
|G:w(G)| & \> = \> \nonumber 
|Hw(G):w(G)| \\ & \> = \>  
|H:H \cap w(G)| \\ & \> \leq \> \nonumber 
  |H:w(H)| \> . \nonumber 
\end{align}

	 \begin{subsection}{Weakly $w$-maximal groups}
	 
	 The next are important properties of weakly $w$-maximal and $w$-maximal groups.
	 It is easy to see that if $G$ is $w$-maximal and $N$ is a normal subgroup contained in $w(G)$, then $G/N$ is $w$-maximal
	 (this is the content of~\cite[Lemma 2.1(b)]{2011GK}).
	 Indeed, a much more general statement is true.
	
	\begin{lemma} \label{lemQuot}
	Every quotient of a weakly $w$-maximal group (resp. $w$-maximal) is weakly $w$-maximal (resp. $w$-maximal).
	\end{lemma}
	\begin{proof}
	Let $G$ be weakly $w$-maximal, and let $N \lhd G$.
	If $N \leqslant H <G$, then Lemma \ref{lemHelp}(i) and some computations show that
	\begin{align*}
|(H/N):w(H/N)| & \> = \> 
|H:Nw(H)| \\ & \> = \> 
\frac{|H:w(H)|}{|N : N \cap w(H)|} \\ & \> \leq \> 
 \frac{|G:w(G)|}{|N : N \cap w(G)|} \\ & \> = \> 
  |G:Nw(G)| \\ & \> = \> 
  |(G/N):w(G/N)| \> .
\end{align*}
If $G$ is $w$-maximal, then the inequality in the middle is strict, and so $G/N$ is $w$-maximal.
	\end{proof}
	
	\begin{lemma} \label{lemDP}
	A direct product $A \times B$ is weakly $w$-maximal (resp. $w$-maximal)
	if and only if both $A$ and $B$ are weakly $w$-maximal (resp. $w$-maximal).
	\end{lemma}
	\begin{proof}
	The ``only if'' part follows from Lemma \ref{lemQuot}.
	For the ``if'' part, let $A$ and $B$ be weakly $w$-maximal, and set $G= A \times B$.
	Let $H < G$.
	Let $\pi : H \rightarrow B$ be the natural projection, and let $H_A$ be the kernel of $\pi$.
	Then, from Lemma \ref{lemHelp}(ii) and the isomorphism theorem, we have
	\begin{equation} \label{eqDP}
	|H:w(H)| \leq |H_A :w(H_A)| |\pi(H):w(\pi(H))| .
	\end{equation}	 
	Since $H_A$ is isomorphic to a subgroup of $A$, we have
	$$ |H:w(H)| \leq |A:w(A)||B:w(B)| = |G:w(G)| , $$
	as desired.
	Finally, let $A$ and $B$ be $w$-maximal groups.
	Now $H$ is proper in $G$, so at least one among the strict inclusions $H_A<A$ and $\pi(H)<B$ is true.
	Hence, we obtain $|H:w(H)|<|G:w(G)|$ from (\ref{eqDP}).
	\end{proof}
	 
	 \end{subsection}

	 \begin{subsection}{$w$-maximal groups}
	 
	 The next simple observation makes $w$-maximal groups a useful tool in the study of finite groups in general.
	 Special cases are implicitly used in~\cite{2011GK,1973Laffey,2022Sab}.
	 
	 \begin{lemma} \label{lemTopApp}
	 Let $G$ be a finite group and let $w$ be a word.
	 Then there exists a $w$-maximal subgroup $H \leqslant G$ such that $|H| \geq |G:w(G)|$.
	 \end{lemma}
	 \begin{proof}
	 It is sufficient to choose $H$ as minimal, with respect to inclusion, with the property that $|H:w(H)| \geq |G:w(G)|$.
	 \end{proof}
	 
	 \begin{remark}
	The main result of~\cite{2022Sab} really gives that every finite group $G$ contains a $\gamma_2$-maximal subgroup
	of size at least $|G|^{O(1)/ \log\log|G|}$.
	It is an important open question whether this can be improved to $|G|^{O(1)}$ for solvable groups~\cite{1999Mann,1997Pyber}.
	\end{remark}
	
	 
	 We emphasize that there exist applications of $w$-maximal groups that are not based solely on Lemma \ref{lemTopApp}.
	 For example, in~\cite{2014GZ} the authors use $w$-maximal groups to give a characterization of powerful $p$-groups, for every $p \geq 5$.
	 We use $\Phi(G)$ to denote the Frattini subgroup of $G$, that is, the intersection of the maximal subgroups of $G$.
	 
	 \begin{lemma} \label{lemFrat}
	 Let $G$ be a $w$-maximal group. Then $w(G) \leqslant \Phi(G)$.
	 \end{lemma}
	 \begin{proof}
	  Suppose that $w(G) \nleqslant \Phi(G)$.
	  Then there exists a proper subgroup $H<G$ which supplements $w(G)$.
	 Therefore, (\ref{eqTop}) gives that $G$ is not $w$-maximal.
	 \end{proof}
	   
	  Following~\cite{2011GK}, a word $w$ is {\itshape interchangeable} in a finite $p$-group $G$ if, for every $N \lhd G$, we have
	  \begin{equation} \label{eqInter}
	  [w(N),G)] \> \leqslant \> [N,w(G)] [w(G),G]^p [w(G),G,G] . 
	  \end{equation}
	  As a consequence of standard commutator formulas,
	  $\gamma_n = [x_1,...,x_n]$ is interchangeable in every finite $p$-group, for each $n \geq 2$ (see~\cite[Lemma 3.1]{2011GK}).
	 In~\cite[Theorem 3.3]{2011GK}, the authors have noticed
	  that (\ref{eqInter}) is the key ingredient to apply Thompson's original argument~\cite{1969Thompson}.
	  
	  \begin{theorem}[Gonz\'alez-S\'anchez and Klopsch] \label{thGK}
	  Let $G$ be a $w$-maximal finite $p$-group. If $w$ is interchangeable in $G$, then $w(G) \leqslant Z(G)$.
	  \end{theorem}
	 
	 \begin{lemma} \label{lemNilpInter}
	 Let $G$ be a $w$-maximal finite group. If $G/w(G)$ is nilpotent, then $G$ is nilpotent.
	 If $G/w(G)$ has nilpotency class $c$, and $w$ is interchangeable in $G$, then $G$ has nilpotency class at most $c+1$.
	 \end{lemma}
	 \begin{proof}
	 From Lemma \ref{lemFrat}, $G/\Phi(G)$ is nilpotent.
	 Then, it is well known that $G$ itself is nilpotent~\cite[2.2.5(b)]{2003KS},
	 i.e. the direct product of $w$-maximal $p$-groups $G_1,...,G_r$ (Lemma \ref{lemDP}).
	 From Remark \ref{remDP}, $G_i/w(G_i)$ is of class at most $c$ for every $i=1,...,r$,
	 and Theorem \ref{thGK} finishes the proof.
	 \end{proof}
	 
	 \end{subsection}

	  \begin{subsection}{Remarks on infinite $w$-maximal groups}
	 
	 In this subsection, let $G$ be a residually finite infinite group, and let $w$ be a word.
	 If $|G:w(G)| = \infty$, then $G$ is weakly $w$-maximal trivially.
	However, it cannot be $w$-maximal: whenever $H \leqslant G$ has finite index, we have $|H:w(H)| \geq \frac{|G:w(G)|}{|G:H|} = \infty$.
	So the interesting case is when $|G:w(G)|<\infty$.
	Here it is important to distinguish whether $w$ is a {\itshape commutator word} (i.e. $w \in [F,F]$) or not.
	The following result improves the direction (1) $\Rightarrow$ (2) of~\cite[Theorem 2.2]{2011GK}, and the proof is much easier.
	We remark that $\gamma_n$ is a commutator word for every $n \geq 2$.
	
	\begin{proposition}
	Let $w$ be a commutator word, and suppose $|G:w(G)|=k<\infty$.
	If $G$ is a residually finite weakly $w$-maximal group, then $G$ is finite.
	Moreover, $|G| \leq k^{O(\log \log k)}$.
	\end{proposition}
	\begin{proof}
	Let $N \lhd G$ be a normal subgroup of finite index $n$, and let $\widetilde{G}=G/N$.
	By~\cite[Theorem 1]{2022Sab}, $\widetilde{G}$ contains an abelian section of size at least $n^{O(1)/\log\log n}$.
	This provides a subgroup $H \leqslant G$ such that
	\begin{align*}
k & \> \geq \> 
|H:w(H)| \\ & \> \geq \> 
|H : \gamma_2(H)|  \\ & \> \geq \> 
  n^{O(1)/\log\log n}  \> .
\end{align*}
	This also gives $\log\log k \geq O(1) \log\log n$, and then
	$$ k^{\log\log k} \geq n^{O(1) \log\log k/\log\log n} \geq n^{O(1)} . $$
	Since a residually finite infinite group has normal subgroups of finite but arbitrarily large index, and $k$ is fixed,
	we achieve that $G$ is finite.
	Therefore, the desired bound is obtained by choosing $N=1$.
	\end{proof}
	
	If $w$ is not a commutator word,
	then there exists an abelian residually finite infinite weakly $w$-maximal group $G$ such that $|G:w(G)|<\infty$.
	The following construction is taken from~\cite[Proof of Theorem 2.2]{2011GK}.
	Let $p$ be a prime and consider a free $\mathbb{Z}_p$-module $G$, where $\mathbb{Z}_p$ denotes the ring of $p$-adic integers.
	Then $w(G)$ is a non-trivial characteristic subgroup of $G$ and hence $w(G) = p^r G$ for some $r \geq 0$.
	In fact, for all submodules $H$ of $G$ we have $w(H) = p^r H$ and consequently $|H:w(H)| \leq |G:w(G)|$.
	 
	 \end{subsection}

	 \begin{subsection}{Nilpotent and solvable quotients}
	 
	 The following special case of Lemma \ref{lemNilpInter} generalizes both results in~\cite{2011GK} and~\cite{2022Sab},
	 where $G$ is a $p$-group and $n=2$ respectively.
	 
	  \begin{lemma} \label{lemNCTop}
	 Let $n \geq 2$. Every $\gamma_n$-maximal finite group is nilpotent of class at most $n$.
	 \end{lemma}
	  
	  We write $\gamma_\infty(G) := \cap_{n \geq 1} \gamma_n(G)$ to denote the nilpotent residual of $G$.
	 
	 \begin{lemma} \label{lemNRes}
	 Let $G$ be a finite group.
	 If $|H:\gamma_\infty(H)| < |G:\gamma_\infty(G)|$ for every $H < G$, then $G$ is nilpotent.
	 \end{lemma}
	 \begin{proof}
	  If $c$ is the nilpotency class of $G/\gamma_\infty(G)$,
	  then the hypotheses of Lemma \ref{lemNCTop} are satisfied with $n=c+1$.
	 \end{proof}
	 
	 \begin{example} \label{exContro}
	  We stress that Lemma \ref{lemNRes} with ``$\leq$'' is false.
	  Let $V=\mathbb{F}_3 \times \mathbb{F}_3$, 
	  and let $K$ be the Sylow $2$-subgroup of $GL(V)$, namely $K$ is a semidihedral group of order $16$.
	  Let $G=V \rtimes K$.
	  Then $\gamma_\infty(G)=V$,
	  and every nilpotent section of $G$ has order bounded by $|G:\gamma_\infty(G)|=|K|$, but $G$ is not nilpotent.
	  Since the nilpotency class of $K$ is $3$, this example also shows that Theorem \ref{thMain} becomes false for each $n \geq 4$.
	  \end{example}
	  
	 In order to deal with the solvable sections, we report~\cite[Proposition 2.2]{1997Pyber}.
	  A subgroup $H \leqslant G$ is said to be {\itshape intravariant} if every image of
	  $H$ under an automorphism of $G$ is conjugate in $G$ to $H$.
	  
	  \begin{proposition}[Chunikhin] \label{propPyb}
	  Let $G=G_0 \rhd ... \rhd G_r = 1$ be a normal series with each $G_i$ normal in $G$. Let $\overline{G_i} = G_{i-1}/G_i$ denote the factors.
	  Suppose for each $i$ that $\overline{H_i}$ is an intravariant subgroup of $\overline{G_i}$.
	  Then there is a subgroup $H$ of $G$ such that all the non-abelian composition factors of $H$
	  occur as composition factors of some of the groups $\overline{H_i}$, and $|H| \geq \prod_i |\overline{H_i}|$.
	  \end{proposition}
	 
	  A very strong result holds for the solvable residual $\mathcal{O}_\infty(G)$,
	  i.e. the intersection of all members of the derived series.
	 Indeed, in the hypotheses of the following lemma,
	 we are just bounding (in the weak sense) the cardinality of the solvable {\itshape subgroups}.
	 
	 \begin{lemma} \label{lemSRes}
	 Let $G$ be a finite group.
	 If $|H| \leq |G:\mathcal{O}_\infty(G)|$ for every solvable subgroup $H$, then $G$ is solvable.
	 \end{lemma}
	 \begin{proof}
	 Suppose that $N=\mathcal{O}_\infty(G) \neq 1$, and consider the normal series $G \rhd N \rhd 1$.
	 Let $P \leqslant N$ be a non-trivial Sylow subgroup of $N$.
	 Since the Sylow $p$-subgroups are all conjugated when $p$ is fixed,
	 we can apply Proposition \ref{propPyb} with $\overline{H_1}=G/N$, and $\overline{H_2}=P$.
	 We obtain a solvable subgroup of $G$ of size at least $|\overline{H_1}| |\overline{H_2}| >|G:N|$.
	 This contradicts the hypotheses.
	 \end{proof}
	 
	 \begin{remark}
	  If $G$ is finite, then $\gamma_\infty(G)$ and $\mathcal{O}_\infty(G)$ are verbal subgroups
	  that coincide with some term of the lower central series, or derived series, of $G$.
	  This is not true for infinite groups, because nilpotent and solvable groups are only {\itshape pseudovarieties}, in the sense of~\cite{1976ES}.
	  \end{remark}
	
	\begin{proof}[Proof of Proposition \ref{propSolvNilpSec}]
	Let $N_0 \lhd G_0 \leqslant G$ be the given solvable section,
	and choose $G_1 \leqslant G_0$ as minimal, with respect to inclusion, with the property that $|G_1:\mathcal{O}_\infty(G_1)| \geq |G_0:N_0|$.
	Then, apply Lemma \ref{lemSRes} to $G_1$.
	The same argument works if $G_0/N_0$ is nilpotent, by using Lemma \ref{lemNRes}.
	\end{proof}
	
	For example, as a consequence of Proposition \ref{propSolvNilpSec}, we can improve an influential theorem of Dixon~\cite{1967Dix}:
	 
	 \begin{theorem} \label{thSym}
	Let $a(n)$ and $b(n)$ denote the largest orders of sections of the symmetric group on $n$ elements,
	which are solvable and nilpotent, respectively.
	Then $a(n) \leq 24^{(n-1)/3}$, and $b(n) \leq 2^{n-1}$.
	 \end{theorem}
	 
	 \end{subsection}
	 
	\end{section}

	\vspace{0.1cm}
	\begin{section}{The Theorem \ref{thMain}} \label{sec2}
	
	In this section, every group is finite.
	Let $w = w(x_1,...,x_n)$ be a word.
	We first observe that, under very mild assumptions on $w$,
	it is sufficient to prove that every weakly $w$-maximal {\itshape solvable} group is nilpotent.
	For this step we use the main result of J.S. Rose~\cite{1977Rose},
	 which describes the insolvable groups with a nilpotent maximal subgroup.
	 The next is clipped from~\cite[Theorem 1]{1977Rose}.
	
	\begin{theorem}[Rose] \label{thRose}
	Let $G$ be a finite non-solvable group with a nilpotent maximal subgroup $M$.
	If $S$ is the unique Sylow $2$-subgroup of $M$ and $Q$ is the unique $2$-complement of $M$, then $Z(Q) \leqslant Z(G)$.
	\end{theorem}
	 
	 We remark that Rose's proof does not require the classification of the finite simple groups~\cite{1994GLS}.
	
	\begin{proposition} \label{propFinalStep}
	  Fix a word $w$. Suppose that the following holds for every weakly $w$-maximal group $G$:
	  \begin{itemize}
	  \item $G/w(G)$ is solvable;
	  \item if $G$ is solvable, then it is nilpotent.
	  \end{itemize}
	  Then every weakly $w$-maximal group is nilpotent.
	  \end{proposition}
	  \begin{proof}
	Let $G$ be a weakly $w$-maximal group.
	Of course, it is sufficient to prove that $G$ is solvable.
	We work by induction on the order of $G$. If $w(G)=1$ we are done, so let $1 \neq w(G)<G$.
	Let $1 \neq N \leqslant w(G)$ be a minimal normal subgroup of $G$.
	Now $G/N$ is weakly $w$-maximal by Lemma \ref{lemQuot}, and then is solvable by induction.
	If $N$ is solvable we are done, so suppose this is not the case.
	Of course, there exists an odd prime $p$ which divides $|N|$. 
	Let $P<G$ be a Sylow $p$-subgroup of $G$.
	Then $P_0=P \cap N$ is a Sylow $p$-subgroup of $N$.
	Since $P_0<N$, the minimality of $N$ provides that $P_0$ is not normal in $G$.
	So there exists a maximal subgroup $M<G$ such that $N_G(P_0) \leqslant M$.
	Therefore, $P \leqslant M$.
	By the Frattini argument we have $G=N N_G(P_0) = NM=w(G)M$.
	Then (\ref{eqTop}) provides $|G:w(G)| \leq |M:w(M)|$.
	Therefore, it is easy to see that $M$ is weakly $w$-maximal:
	if $H \leqslant M < G$, then
	$$ |H:w(H)| \leq |G:w(G)| \leq |M:w(M)| . $$
	By induction, $M$ is solvable and so is nilpotent by assumption.
	With the notation of Theorem \ref{thRose}, we obtain $M=S \times Q$, with $Z(Q) \leqslant Z(G)$.
	Moreover, $P_0 \lhd P \leqslant Q$ because $p \neq 2$. 
	It follows that
	$$ 1 \neq P_0 \cap Z(P) \leqslant N \cap Z(Q) \leqslant N \cap Z(G) , $$
	where we also used that $P$ is a direct factor of $Q$.
	By the minimality of $N$ we obtain $N \leqslant Z(G)$, against the assumption that $N$ is not solvable.
	\end{proof}
	 
	We now explain the relations between the notions of (weak) $\gamma_n$-maximality and (weak) $\gamma_{n+1}$-maximality.
	 In particular, we remark that Theorem \ref{thMain} is really giving two different results.
	
	\begin{example}
	Let $n \geq 2$.
	It is trivial to see that every group of nilpotency class $n$ is $\gamma_{n+1}$-maximal,
	but not even weakly $\gamma_n$-maximal in general
	(for example, \texttt{SmallGroup(32,31)}~\cite{2011Gap} is of class $2$, has commutator subgroup of size $4$, and an abelian maximal subgroup).
	On the other hand, by Lemma \ref{lemNCTop}, every $\gamma_n$-maximal group is $\gamma_{n+1}$-maximal.
	Finally, there exist weakly $\gamma_2$-maximal groups which are not weakly $\gamma_3$-maximal
	(see Example \ref{exContro2} in the next subsection).
	\end{example}
	
	We will use the following theorem of T.M. Keller and Y. Yang~\cite{2018KY},
	which improves an older result of G. Glauberman~\cite{1975Glau}.
	
	\begin{theorem}[Theorem 1.1 in~\cite{2018KY}] \label{thKY}
	Let $p$ be a prime and let $V$ be a non-trivial elementary abelian $p$-group.
	Suppose $G$ is a solvable $p'$-group of automorphisms of $V$.
	Then $|G:\gamma_3(G)| < |V|$.
	\end{theorem}
	 
	 The strict inequality in Theorem \ref{thKY} (already present in~\cite[Proposition 1]{1975Glau}) follows from arithmetic considerations.
	 It is the reason behind the validity of Theorem \ref{thMain}, as we are about to see.
	 
	 \begin{proof}[Proof of Theorem \ref{thMain}]
	 Fix $n=2$ or $n=3$.
	 From Proposition \ref{propFinalStep} and Lemma \ref{lemDP},
	 it is sufficient to prove that every solvable weakly $\gamma_n$-maximal group is nilpotent.
	 Let $G$ be a counterexample of minimal order.
	From Lemma \ref{lemQuot} and the minimality of $G$, every quotient of $G$ is weakly $\gamma_n$-maximal and so nilpotent.
	Let $M<G$ be a non-normal maximal subgroup. If $M$ has a non-trivial normal core $C$,
	then $M/C \lhd G/C$ implies the contradiction $M \lhd G$.
	Now let $1 \neq N \lhd G$ be a minimal normal subgroup.
	Since $G$ is solvable, $N$ is an elementary abelian $p$-group for some prime $p$.
	Moreover, $NM=G$ and $N \cap M \lhd NM=G$ (because $N \cap M$ is normal in both $N$ and $M$).
	By the minimality of $N$ we obtain $N \cap M=1$, and so $G= N \rtimes M$.
	Also, as the centralizer $C_M(N)$ is contained in the (trivial) normal core of $M$, it follows that $M \leqslant \Aut(N)$.
	Furthermore, $M$ is a $p'$-group:
	it is nilpotent, and it is well known that in our situation we have $O_p(M)=1$ (see ~\cite[6.6.3(a)]{2003KS}).
	We also remark that $N$ is the unique minimal normal subgroup of the affine primitive group $G$,
	so that $N \leqslant \gamma_n(G)$.
	This gives $\gamma_n(G)=N \rtimes \gamma_n(M)$.
	Therefore $|G:\gamma_n(G)| = |M:\gamma_n(M)|$,
	 but applying Theorem \ref{thKY} we obtain $|M:\gamma_n(M)| <|N|$.
	Since $N$ is an abelian subgroup of $G$, this contradicts the fact that $G$ is weakly $\gamma_n$-maximal.
	 \end{proof}
	
	\begin{remark}
	 Bounding the order of the nilpotent {\itshape subgroups} of a certain class is not sufficient to obtain nilpotency.
	 In the case $w=\gamma_2$, the \texttt{SmallGroup(96,201)}~\cite{2011Gap} is a non-nilpotent group with abelianization of order $12$,
	 whose abelian subgroups have order at most $12$.
	 A related question is how, in general, $|H:\gamma_2(H)|$ behaves with respect to $|G:H|$.
	 We refer the reader to~\cite{2023Sab} for asymptotic results on this problem,
	  and some connections to representation theory.
	 \end{remark}

	\begin{subsection}{Large nilpotency class} \label{subsect2.2}
	
	Let $p$ be a prime and let $n \geq 2$. Let $N$ be the cyclic group of size $p^n$.
    Then $\Aut(N)$ is an abelian group of order $p^{n-1}(p-1)$ (see ~\cite[2.2.5]{2003KS}).
    Let $K$ be the Sylow $p$-subgroup of $\Aut(N)$, and consider $P = N \rtimes K$.
    In this subsection, we show that $P$ is a weakly $\gamma_2$-maximal group of class $n$.
    We will frequently write $P'$ in place of $\gamma_2(P)$.
    Moreover, given $x,y \in P$, we set $x^y:=y^{-1}xy$, and $[x,y]:=x^{-1}x^y$.
    
    \begin{lemma} \label{lemLow}
    $|P:P'|=p^n$, and $|\gamma_i(P):\gamma_{i+1}(P)|=p$ for every $i=2,...,n$.
    \end{lemma}
    \begin{proof}
    We first assume $p>2$, or $p=n=2$.
    Let $x$ be a generator of $N$ and let $\sigma$ be a generator of $K$.
    Since $N$ and $K$ are abelian, we have $P'=[N,K]$, and $\gamma_i(P)=[\gamma_{i-1}(P),K]$ for every $i \geq 3$.
    By~\cite[2.2.6(a)]{2003KS}, we can assume $x^\sigma=x^{p+1}$.
    Then $[x^j,\sigma]= x^{-j} x^{(p+1)j}  = x^{pj}$ for every $j=0,...,p^n-1$.
    This implies that $[N,K]=\langle x^p \rangle$,
    and $\gamma_i(P)=[\gamma_{i-1}(P),K]=\langle x^{p^{i-1}} \rangle$ for every $i \geq 3$.
    Then $|N:\gamma_i(P)| = |\langle x \rangle : \langle x^{p^{i-1}} \rangle|=p^i$ for every $i=2,...,n+1$.\\
    If $p=2<n$ then $K$ is not cyclic,
    but contains elements $\sigma$ and $\tau$ such that $x^\sigma=x^5$ and $x^\tau=x^{-1}$~\cite[2.2.6(b)]{2003KS}.
    We remark that
    \begin{equation} \label{eqP2}
    [x,\sigma \tau] = [x,\tau] [x,\sigma]^\tau = x^{-2} x^{-4} = x^{-6} .
    \end{equation}
    This implies $P'= \langle x^6 \rangle = \langle x^2 \rangle$.
    Furthermore, if $y_1=x^2$ and $y_{i+1}=[y_i,\sigma \tau]$ for every $i \geq 1$,
    then $\gamma_{i}(P)=\langle y_i \rangle$, and the proof follows.
    \end{proof} 
    
   \begin{lemma}
   $P$ is weakly $\gamma_2$-maximal.
   \end{lemma}
   \begin{proof}
   We first assume $p>2$.
   Let $H \leqslant P$.
   We have to prove $|H:H'| \leq p^n$, so we can assume $|H| \geq p^{n+1}$.
   Then
   $$ |N \cap H|=\frac{|N||H|}{|NH|} \geq \frac{p^n \cdot p^{n+1}}{p^{2n-1}} = p^2 . $$
	Let $R=NH \cap K$. Since $N$ and $R$ are abelian, we have
	\begin{align*}
[N \cap H,R] & \> = \> 
[N \cap H,NR] \\ & \> = \> 
 [N \cap H,NH] \\ & \> = \> 
  [N \cap H,H] \\ & \> \leqslant \> 
  H' \> . 
\end{align*}
	We are going to estimate $|[N \cap H,R]|$.
	Let $x$ be a generator of $N \cap H$.
	As in the proof of Lemma \ref{lemLow}, let $\sigma$ be a generator of $K$.
	If $|K:R|=p^a$ for some $a \geq 0$, then $\sigma^{(p^a)}$ is a generator of $R$.
	Of course, $[N \cap H,R]$ contains the element
	$$ [x,\sigma^{(p^a)}] = x^{-1} x^{\sigma^{(p^a)}} = x^{-1} x^{(p+1)^{(p^a)}} . $$
	Now $(p+1)^{(p^a)} = \sum_{i=0}^{p^a} {p^a \choose i} p^i$,
	and ${p^a \choose i}$ is divisible by $p^a$ for each $1 \leq i < p^a$.
	It follows that
	$$ (p+1)^{(p^a)} \equiv p^{a+1} +1 \pmod {p^{a+2}} . $$
	This implies $\langle x^{p^{a+1}} \rangle \leqslant [N \cap H,R]$,
	and so $|N \cap H : [N \cap H,R]| \leq p^{a+1} = |K:R|p$.
	Hence we can write
	\begin{align*}
|H:H'| & \> \leq \> 
\frac{|H|}{|[N \cap H,R]|} \\ & \> = \> 
 \frac{|H|}{|N \cap H|} \cdot \frac{|N \cap H|}{|[N \cap H,R]|} \\ & \> \leq \> 
  |R| |K:R|p \\ & \> = \> 
  p^n \> . 
\end{align*}
The case $p=2$ is similar, by using (\ref{eqP2}).
   \end{proof}
	  
	  \begin{example} \label{exContro2}
	  The groups in this subsection also offer weakly $\gamma_2$-maximal groups that are not weakly $\gamma_3$-maximal.
	  Let $p$ be any prime, and choose $n=4$ and $P$ as in the previous construction.
	  Then $P \cong C_{p^4} \rtimes C_{p^3}$, and $|P:\gamma_3(P)|=p^5$.
	  On the other hand, it is easy to check that $P$ has a maximal subgroup of nilpotency class $2$.
	  \end{example}
	 
	\end{subsection}
	
	\end{section}

	\vspace{0.1cm}
	\begin{section}{Isologisms and $\gamma_2$-maximal groups} \label{sec3}
	 
	 \begin{subsection}{Isologisms} \label{SubSecIso}
	 
	 For the convenience of the reader,
	 let us explain some terminology from the theory of isologisms,
	 the standard references being~\cite{1940Hall,1940HallB,1989Hek}.
	Fix a word $w=w(x_1,...,x_n)$.
	In every group $G$, we have the word map
	$$ w : \> \overbrace{G \times ... \times G}^n \> \longrightarrow \> G . $$
	The {\itshape marginal subgroup} $m_w(G)$ for $w$ in $G$ is the set of elements $g \in G$ such that
	$$ w(x_1,...,x_n) \> = \> w(x_1,...,x_i g,...,x_n) \> = \> w(x_1,..., g x_i,...,x_n) $$
	for every $x_1,...,x_n \in G$ and every $i=1,...,n$.
	The prototype is $w=\gamma_2=[x_1,x_2]$, where $w(G)=G'$ and $m_w(G)=Z(G)$.
	We just write $m(G)=m_w(G)$ when $w$ is clear from the context.
	By definition of $m(G)$, the map $w$ induces a map
	$$ \theta_{w} : \> \overbrace{\frac{G}{m(G)} \times ... \times \frac{G}{m(G)}}^n \> \longrightarrow \> w(G) . $$
	Now, two groups $A$ and $B$ are said to be {\itshape $w$-isologic} if there exist isomorphisms
	$$ \phi : w(A) \rightarrow w(B) \hspace{1cm} \mbox{and} \hspace{1cm} \psi : A/m(A) \rightarrow B/m(B) $$
	which commute with $\theta_w$.
	Isologism is an equivalence relation, and divides the groups into families~\cite{1940HallB}.
	  The content of Theorem \ref{thIso} is: if two finite groups $A$ and $B$ are $w$-isologic,
	  then one among them is weakly $w$-maximal (resp. $w$-maximal) if and only if both are weakly $w$-maximal (resp. $w$-maximal).
	
	\begin{proof}[Proof of Theorem \ref{thIso}]
	Let $A$ and $B$ be $w$-isologic finite groups, and let $A$ be weakly $w$-maximal.
	Let $\phi: w(A) \rightarrow w(B)$ and $\psi: A/m(A) \rightarrow B/m(B)$ be the isomorphisms in the definition of $w$-isologic groups.
	By the definition of $m(B)$, for every $H \leqslant B$, we have $w(Hm(B))=w(H)$.
	 So, it is enough to prove the inequality $|H:w(H)| \leq |B:w(B)|$ for the subgroups $H$ that contain $m(B)$.
	Let $m(B) \leqslant H < B$.
	Moreover, let $m(A) \leqslant \overline{H} < A$ such that $\overline{H}/m(A) = \psi(H/m(B))$.
	We have
	\begin{align*}
|H:w(H)| & \> = \> 
\frac{|H/m(B)||m(B)|}{|w(H)|} \\ & \> = \> 
\frac{|\psi(H/m(B))| |m(B)|}{|\phi(w(H))|} \\ & \> = \> 
  \frac{|\overline{H}| |m(B)| }{|\phi(w(H))| |m(A)|} \> .
\end{align*}
	Since $\phi$ and $\psi$ commute with the map $\theta_w$, it is easy to check that $\phi(w(H))=w(\overline{H})$.
	This implies $|\overline{H}:\phi(w(H))| \leq |A:w(A)|$, and so from the previous equality we can write
	\begin{align*}
|H:w(H)| & \> \leq \> 
|A:w(A)| \cdot \frac{|m(B)|}{|m(A)|} \\ & \> = \> 
|A:m(A)| \cdot \frac{|m(B)|}{|w(A)|} \\ & \> = \> 
 |B:m(B)| \cdot \frac{|m(B)|}{|w(B)|} \\ & \> = \> 
  |B:w(B)| \> .
\end{align*}
	If $A$ is $w$-maximal, then the first inequality is strict, and so $B$ is $w$-maximal.
	\end{proof}
	 
	 Certainly, Theorem \ref{thIso} asks for a classification of $w$-maximal groups up to $w$-isologism.
	 We hope to stimulate further research in this direction.
	  
	 \end{subsection}

	\begin{subsection}{$\gamma_2$-maximal groups}
	
	We conclude this article with some questions regarding (weakly) $\gamma_2$-maximal finite groups.
	These have not been investigated since Thompson's paper~\cite{1969Thompson},
	and Theorem \ref{thIso} can help to understand their structure.
	After Theorem \ref{thMain}, we are really interested in the $p$-group case.
	We start with the following
	
	 \begin{question}
	Is the derived length of weakly $\gamma_2$-maximal finite groups bounded?
	 \end{question}
	  
	  If there exists some positive integer $n$ such that the derived length is at most $n$,
	  then every weakly $\gamma_2$-maximal finite group $G$ has an abelian section of size at least $|G|^{1/n}$ (which is located in the derived series).
	  In particular, by the weak $\gamma_2$-maximality we may have $|G:G'| \geq |G|^{1/n}$.
	
	\begin{question} \label{conj1}
	 Does there exist $\delta>0$ such that every weakly $\gamma_2$-maximal finite group $G$ satisfies $|G:G'| \geq |G|^\delta$?
	 \end{question}
	  
	  We point out that Question \ref{conj1} is a weak version of~\cite[Problem 14.76]{2022KM},
	  which asks for a large abelian section in every finite $p$-group.
	  The examples of Subsection \ref{subsect2.2} provide $|G:G'|=(|G| \cdot p)^{1/2}$,
	  so the $\delta$ in Question \ref{conj1} cannot be larger than $1/2$.
	  Since the center of a non-abelian group contains properly some abelian subgroup,
	  every $\gamma_2$-maximal finite group $G$ satisfies $|G:G'| > |G|^{1/2}$.
	  The known examples suggest the following
	
	\begin{question} \label{conj2}
	 Does there exist $\delta>1/2$ such that every $\gamma_2$-maximal finite group $G$ satisfies $|G:G'| \geq |G|^\delta$?
	 \end{question}
	  
	  We recall that $\gamma_2$-isologisms are called {\itshape isoclinisms}, and divide $p$-groups into families~\cite{1940Hall}.
	The smallest members of each family satisfy $Z(G) \leqslant G'$
	(these are the so-called {\itshape stem groups}, see~\cite[pag. 135]{1940Hall}).
	  Using Theorem \ref{thIso}, we can reduce all the three problems above to the case of stem groups,
	  and in particular to $Z(G)=G'$ for $\gamma_2$-maximal groups.
	  First, the derived length is invariant under isoclinism, because isoclinic groups have isomorphic commutator subgroups.
	  Moreover, suppose one has a positive answer to Question \ref{conj1},
	  for every weakly $\gamma_2$-maximal finite group $W$ with $Z(W) \leqslant W'$.
	   If $G$ is any weakly $\gamma_2$-maximal finite group which is isoclinic to $W$,
	  then
	  \begin{align*}
|G:G'| & \> = \> 
\frac{|G|}{|W'|} \\ & \> \geq \> 
|W:W'|  \\ & \> \geq \> 
|W'|^{\delta/(1-\delta)} \\ & \> = \> 
 |G'|^{\delta/(1-\delta)} \> ,
\end{align*}
	  which in turn means $|G:G'| \geq |G|^\delta$.
	  The same argument works for Question \ref{conj2}.
	 
	 \end{subsection}
	
	\end{section}

	\vspace{0.1cm}

\end{document}